\documentclass[12pt,notitlepage,twoside,a4paper]{amsart}
 \usepackage{amsfonts}

\usepackage{amsmath,amssymb,enumerate}

\usepackage{epsfig,fancyhdr,color}

\usepackage{amssymb}
\usepackage{amsmath,amsthm}
\usepackage{latexsym}
\usepackage{amscd}
\usepackage{psfrag}
\usepackage{graphicx}
\usepackage[latin1]{inputenc}
\usepackage[all]{xy}
\usepackage[mathcal]{eucal}

\definecolor{NoteColor}{rgb}{1,0,0}

\renewcommand{\textsc}{\textcolor{red}}

\newtheorem{theorem}{\rm\bf Theorem}[section]
\newtheorem{proposition}[theorem]{\rm\bf Proposition}
\newtheorem{lemma}[theorem]{\rm\bf Lemma}
\newtheorem{construction}[theorem]{\rm\bf Construction}

\newtheorem*{theorem 1}{\rm\bf Proposition 1}
\newtheorem*{theorem 2}{\rm\bf Proposition 2}

\theoremstyle{definition}

\theoremstyle{remark}
\newtheorem{remark}[theorem]{\rm\bf Remark}

\newtheorem{problem}[theorem]{\rm\bf Problem}

\def\interieur#1{\mathord{\mathop{\kern 0pt #1}\limits^\circ}}

\title[Hyperbolic geometry]{On hyperbolic analogues of some classical theorems in spherical geometry}

\author{Athanase Papadopoulos}
\address{Athanase Papadopoulos,   Institut de Recherche Math\'ematique Avanc\'ee, Universit{\'e} de Strasbourg and CNRS,
7 rue Ren\'e Descartes,
 67084 Strasbourg Cedex, France ;
 and CUNY, Hunter College, Department of Mathematics  and Statistics,  695, park Ave. NY 10065, USA} 
 \email{athanase.papadopoulos@math.unistra.fr}

\author{Weixu Su}\address{Weixu Su,   Department of Mathematics, Fudan University, 200433, Shanghai, China} 
 \email{suwx@fudan.edu.cn}

\date{\today}

\begin{document}

\maketitle
  
\begin{abstract} We provide hyperbolic analogues of some classical theorems in spherical geometry due to Menelaus,  Euler, Lexell,  Ceva and  Lambert. Some of the spherical results are also made more precise. Our goal is to go through the works of some of the eminent mathematicians from the past and to include them in a modern perspective. Putting together results in the three constant-curvature geometries and highlighting the analogies between them is mathematically as well as aesthetically very appealing.

\medskip

\noindent The paper will appear in a volume of the Advanced Studies in Pure Mathematics, ed. L. Fujiwara, S. Kojima and K. Ohshika, Mathematical Society of Japan.

\medskip


\medskip

\end{abstract}

\bigskip


\section{Introduction}

We obtain hyperbolic analogues of several theorems in spherical geometry. The first theorem is due to Menelaus and is contained in his \emph{Spherics} (cf. \cite{RP1} \cite{RP2} \cite{Krause} \cite{RP}). The second is due to Euler \cite{Euler-Geometrica-T}. The third was obtained by Euler \cite{Euler-Variae-T} and by his student Lexell \cite{Lexell-Solutio}. We shall elaborate in the corresponding sections on the importance and the impact of each of these theorems. We also include a  proof of the hyperbolic version of the Euclidean theorem attributed to Ceva, which is in the same spirit as Euler's theorem (although the proof is easier). We also give a proof of the hyperbolic version of a theorem of Lambert, as an application of the hyperbolic version of the theorem of Euler that we provide. In the course of proving the hyperbolic analogues, we also  obtain more precise versions of some of the results in spherical geometry. 

\medskip

\noindent {\bf Acknowledgements.} The second author is partially supported by the NSFC grant No: 11201078. Both authors are partially supported by the French ANR grant FINSLER. They are thankful for Norbert A'Campo discussions on this subject and to the referee for reading carefully the manuscript and making several corrections.

\section{A result on right triangles}

We start with a result on right triangles which makes  a relation between the hypothenuse and a cathetus, in terms of the angle they make (Theorem \ref{thm:ratio}). This is a non-Euclidean  analogue of the fact that in the Euclidean case, the ratio of the two corresponding lengths is the cosine of the angle they make. Our result in the hyperbolic case is motivated by a similar (but weaker) result of Menelaus\footnote{Since this paper is motivated by classical theorems, a few words on history are in order. We have included them, for the interested reader, in this footnote and the following ones. We start in this note by some notes on the works of  Theodosius (2nd-1st c. B.C.), Menelaus (1st-2nd c. A.D.) and Euler (1707-1783).

Anders Johan Lexell (1740--1784), who was a young collaborator of Euler at the Saint-Petersburg Academy of Sciences and who was very close to him, concerning the work done before the latter on spherical geometry,  mentions Theodosius. He writes in the introduction to his paper \cite{Lexell-Solutio}: ``From that time in which the Elements of Spherical Geometry of Theodosius had been put on the record, hardly any other questions are found, treated by the geometers, about further perfection of the theory of figures drawn on spherical surfaces, usually treated in the Elements of Spherical Trigonometry  and aimed to be used in the solution of spherical triangles." A French translation of Theodosius' \emph{Spherics} is available \cite{Theodosius}.

 The work of Menelaus, which was done two centuries after Theodosius, is, in many respects, superior to the work of Theodosius. One reason is the richness and the variety of the results proved, and another reason is the methods used, which are intrinsic to the sphere. These methods  do not make use of the geometry of the ambient Euclidean space. No Greek manuscript of the important work of Menelaus, the \emph{Spherics},  survives, but manuscripts of Arabic translations are available (and most of them are still not edited). For this reason, this work is still very poorly known even today, except for the classical ``Menelaus Theorem" which gives a condition for the alignment of three points situated on the three lines containing the sides of a triangle. This theorem became a classic because it is quoted by Ptolemy, who used it in his astronomical major treatise, the \emph{Almagest}. Lexell and Euler were not aware of the work of Menelaus, except for his results quoted by Ptolemy. A critical edition with an English translation and mathematical and historical commentaries of  al  Haraw\=\i  's  version of Menelaus' \emph{Spherics}  (10th c.)  is in preparation \cite{RP}. We note by the way that the non-Euclidean versions of Menelaus' Theorem were used recently in the papers \cite{2012-Yamada2} \cite{2012-Yamada1}, in a theory of the Funk and Hilbert metrics on convex sets in the non-Euclidean spaces of constant curvature. This is to say that putting classical theorems in a modern perspective may be imporant for research conducted today.

Between the times of Menelaus and of Euler, no progress was made in the field of spherical geometry. Euler wrote  twelve papers on spherical geometry and in fact he revived the subject. Several of his young collaborators and disciples followed him in this field (see the survey \cite{AP}).

}  in the spherical case contained in his \emph{Spherics}. Menelaus' result is of major importance from the historical point of view, because the author gave only a sketch of a proof, and writing a complete proof of it gave rise to several mathematical developments by Arabic mathematicians between the 9th and the 13th centuries. (One should remember that the set of spherical trigonometric formulae that is available to us today and on which we can build our proofs was not available at that time.) These developments include the discovery of duality theory and in particular the definition of the polar triangle in spherical geometry, as well as the  introduction of an invariant spherical cross ratio. It is also probable that the invention of the sine rule was motivated by this result. All this is discussed in the two papers \cite{RP1} and \cite{RP2}, which contains a report on the proof of Menelaus' theorem completed by several Arabic mathematicians. 

The proof that we give of the hyperbolic version of that theorem works as well in the spherical case, with a modification which, at the formal level, amounts to replacing some hyperbolic functions by the corresponding circular functions. (See Remark \ref{rem} at the end of this section.) Thus, in particular, we get a very short proof of Menelaus' Theorem.

\begin{figure}[htbp]

\centering
\includegraphics[width=10cm]{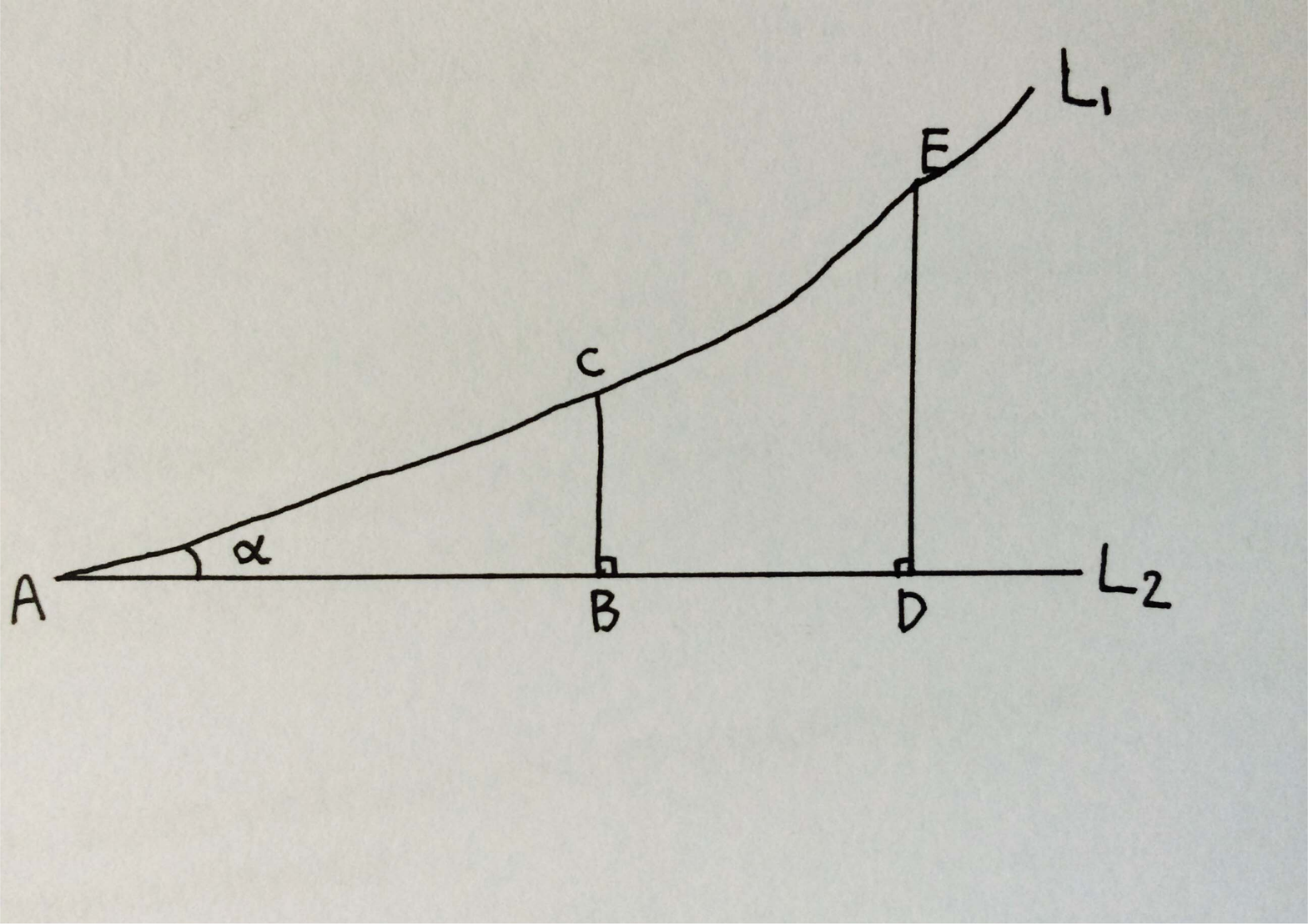}
\caption{The right triangles $ABC$ and $ADE$.} \label{fig:right}
\end{figure}

The statement of this theorem refers to Figure \ref{fig:right}.

\begin{theorem}\label{thm:ratio}
In the hyperbolic plane, consider two geodesics $L_1,L_2$ starting at a point $A$ and making an acute angle $\alpha$. Consider two points $C$ and $E$ on $L_1$, with $C$ between $A$ and $E$, and the two perpendiculars $CB$ and $ED$ onto $L_2$. Then, we have:

$$\frac{\sinh (AC+AB)}{\sinh (AC-AB)}=\frac{1+\cos \alpha}{1-\cos\alpha}.$$
\end{theorem}

 In particular, we have
\[
\frac{\sinh (AC+AB)}{\sinh (AC-AB)}= \frac{\sinh (AE+AD)}{\sinh (AE-AD)},
\]
which is the form in which Menelaus stated his theorem in the spherical case (where $\sinh$ is replaced by $\sin$).

   To prove Theorem \ref{thm:ratio}, we use the following lemma.

\begin{lemma}\label{lem:ratio} In the triangle $ABC$, let $a=BC$, $b=AC$ and $c=AB$. Then we have:
$$\tanh c =\cos \alpha \cdot  \tanh b.$$
\end{lemma}
\begin{proof}
The formula is a corollary of the cosine and sine laws for hyperbolic triangles. We provide the complete proof.

From the hyperbolic cosine law, we have (using the fact that the angle $\widehat{ABC}$ is right)
$$\cosh b= \cosh a \cdot \cosh c.$$
By the hyperbolic sine law, we have
$$\sinh b = \frac{\sinh a}{\sin \alpha}.$$
As a result, we have

\begin{eqnarray*}
\sin^2\alpha &=& \frac{(\sinh a)^2}{(\sinh b)^2} = \frac{(\cosh a)^2-1}{(\sinh b)^2} \\
  &= & \displaystyle \frac{\frac{(\cosh b)^2}{(\cosh c)^2}-1}{(\sinh b)^2}= \frac{(\cosh b)^2-(\cosh c)^2}{(\cosh c)^2 \cdot (\sinh b)^2}.
\end{eqnarray*}
Then we have

\begin{eqnarray*}
\cos^2\alpha &=& 1- \frac{(\cosh b)^2-(\cosh c)^2}{(\cosh c)^2 \cdot (\sinh b)^2} \\
&=& \frac{(\cosh c)^2 \cdot (\sinh b)^2-(\cosh b)^2+(\cosh c)^2}{(\cosh c)^2 \cdot (\sinh b)^2}\\
&=& \frac{(\cosh c)^2 \cdot (\cosh b)^2-(\cosh b)^2}{(\cosh c)^2 \cdot (\sinh b)^2} \\
&=& \frac{(\sinh c)^2 \cdot (\cosh b)^2}{(\cosh c)^2 \cdot (\sinh b)^2}.
\end{eqnarray*}
Since $\cos\alpha >0$, we get
$$\cos\alpha= \frac{\sinh c \cdot \cosh b}{\cosh c \cdot \sinh b}.$$
\end{proof}

To prove Theorem \ref{thm:ratio}, it suffices to write the ratio $\displaystyle \frac{\sinh (b+c)}{\sinh (b-c)}$
as
$$\frac{\sinh b \cdot \cosh c + \cosh b \cdot \sinh c}{\sinh b \cdot \cosh c-\cosh b \cdot \sinh c}=\frac{1+\frac{\tanh c}{\tanh b}}{1-\frac{\tanh c}{\tanh b}}.$$
Using Lemma \ref{lem:ratio}, the above ratio becomes
$$\frac{1+\cos\alpha}{1-\cos\alpha}.$$

\begin{remark}\label{rem}

An analogous proof works for the spherical case, and it gives the following more precise result of Menelaus' theorem:
$$\frac{\sin (AC+AB)}{\sin (AC-AB)}=\frac{1+\cos \alpha}{1-\cos\alpha}.$$

\end{remark}

\section{Euler's ratio-sum formula for hyperbolic  triangles}
Euler, in his memoir \cite{Euler-Geometrica-T},\footnote{The memoir was published in 1815, that is, 22 years after Euler's death. There was sometimes a long span of time between the moment Euler wrote his articles and the moment they were published. The main reason was Euler's unusual productivity, which caused a huge backlog in the journals of the two Academies of Sciences (Saint Petersburg and Berlin) where he used to send his articles. There were also other reasons. For instance, it happened several times that when Euler knew that another mathematician was working on the same subject, he intentionally delayed the publication of his own articles, in order to leave the priority of the discovery to the other person, especially when that person was a young mathematician, like, in the present case, Lexell. Another instance where this happened was with Lagrange, who was 19 years younger than Euler, and to whom the latter, at several occasions, left him the primacy of publication, for what concerns the calculus of variations  ;  see \cite{AP}  and \cite{AP1} for comments on the generosity of Euler.} proved the following:

\begin{theorem}\label{thm:Euler}
Let $ABC$ be a triangle in the plane and let $D, E,F$ be points on the sides $BC, AC,AB$ respectively.
If the lines $AD, BE,CF$ intersect at a common point $O$, then we have
\begin{equation}\label{equ:Euler}
\frac{AO}{OD}\cdot\frac{BO}{OE}\cdot\frac{CO}{OF}=\frac{AO}{OD}+\frac{BO}{OE}+ \frac{CO}{OF}+2.
\end{equation}
\end{theorem}

The following notation will be useful for generalization: Setting $\alpha=\displaystyle \frac{AO}{OD},\beta=\frac{BO}{OE}, \gamma=\frac{CO}{OF}$, Equation $\eqref{equ:Euler}$  is equivalent to \begin{equation}\label{equ:Euler1}
\alpha\beta\gamma=\alpha+\beta+\gamma+2.
\end{equation}

Euler also gave the following construction which is a converse of Theorem \ref{thm:Euler}:
  \begin{construction} 
Given three segments $AOD, BOE, COF$ meeting at a common point $O$ and satisfying $\eqref{equ:Euler}$, we can construct a triangle $ABC$ such that the points $D,E,F$ are as in the theorem.
\end{construction}

\begin{figure}[htbp]

\centering

\includegraphics[width=10cm]{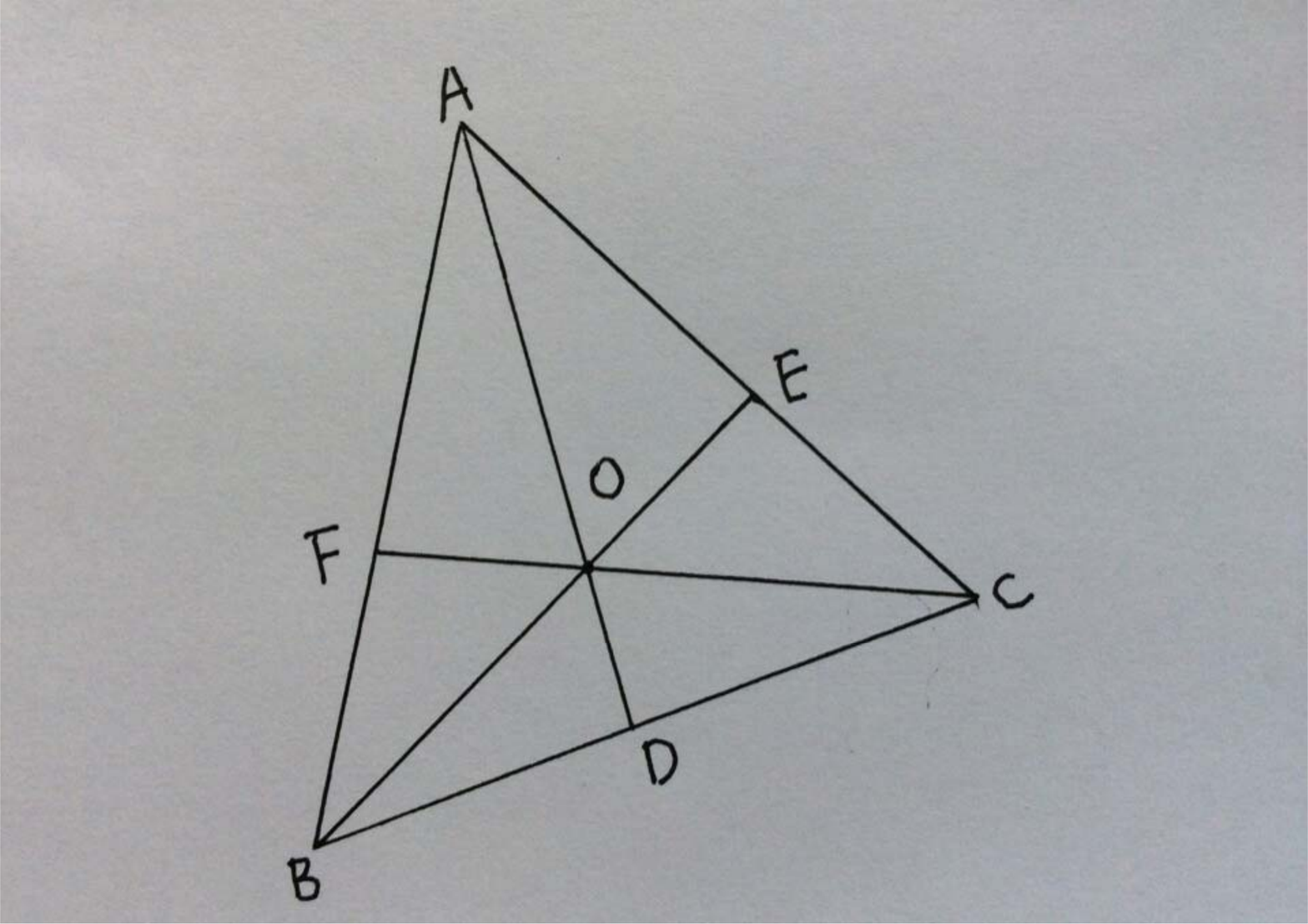}
\label{fig:triangle}
\caption{Triangle.} \label{fig:triangle}

\end{figure}

We now recall Euler's spherical version. We use the angular metric on the sphere. In the following exposition, in order to avoid dealing with special cases which would distract from the general theory, we shall assume that all the spherical triangles that we consider are contained in a quarter plane. In particular, the length of each side of a triangle is $\leq \pi/2$.

After the Euclidean case, Euler proved a version of Theorem \ref{thm:Euler} for spherical triangles. Using the above notation, we state Euler's theorem:

\begin{theorem}\label{thm:Euler2}
Let $ABC$ be a spherical triangle and let $D, E,F$ be points on the sides $BC, AC,AB$ respectively.
If the lines $AD, BE,CF$ intersect at a common point $O$, then 
\begin{equation}\label{equ:Euler2}
\alpha\beta\gamma=\alpha+\beta+\gamma+2
\end{equation}
where  $\alpha=\displaystyle \frac{\tan AO}{\tan OD},\ \beta=\frac{\tan BO}{\tan OE}$ and $\displaystyle \gamma=\frac{\tan CO}{\tan OF}$.
\end{theorem}

We now prove an analogous result for hyperbolic triangles:

\begin{theorem}\label{thm:hyperbolic}
Let $ABC$ be a triangle in the hyperbolic plane and let $D, E,F$ be points on the lines joining the sides $BC, AC,AB$, respectively.
If the lines $AD, BE,CF$ intersect at a common point $O$, then
\begin{equation}\label{equ:hyperbolic}
\alpha\beta\gamma=\alpha+\beta+\gamma+2,
\end{equation}
where  $\alpha=\displaystyle \frac{\tanh AO}{\tanh OD},\, \beta=\frac{\tanh BO}{\tanh OE}$ and $\displaystyle\gamma=\frac{\tanh CO}{\tanh OF}$.
\end{theorem}

\subsection{Proof of Theorem \ref{thm:hyperbolic}}

Let $ABC$ be a triangle in the hyperbolic plane.
Suppose that the lines $AD, BE$ and $CF$ intersect at $O$.

We shall use the cosine and sine laws in the hyperbolic triangle $AFO$.
The first formula gives the cosine of the angle  $\widehat{AFO}$ in terms of the side lengths of that triangle:
\begin{equation}\label{equ:cosine}
\cos \widehat{AFO}= \frac{\cosh AF \cdot \cosh OF -\cosh AO}{\sinh AF \cdot \sinh OF}.
\end{equation}

\begin{equation}\label{equ:sine}
\frac{\sinh AF}{\sin \widehat{AOF}}=\frac{\sinh AO}{\sin \widehat{AFO}}.
\end{equation}

\begin{lemma}\label{lem:sin} We have
$$\sin \widehat{BOD} = \tanh OF \cdot \left( \frac{\sin\widehat{BOF}}{\tanh AO} +  \frac{\sin\widehat{AOF}}{\tanh BO} \right).$$
\end{lemma}
\begin{proof}
Combining $\eqref{equ:cosine}$  with $\eqref{equ:sine}$, we have
\[
\tan \widehat{AFO} = \frac{\sin \widehat{AFO}}{\cos \widehat{AFO}}=\frac{\sinh AO \cdot \sinh OF\cdot  \sin\widehat{AOF}}{\cosh AO\cdot \cosh OF-\cosh AF} 
\]
\smaller
\[= \frac{\sinh AO \cdot \sinh OF\cdot  \sin\widehat{AOF}}{\left( \cosh AO \cdot \cosh OF-\sinh AO\cdot\sinh OF\cdot \cos \widehat{AOF}\right)\cdot\cosh OF-\cosh AO},
\]
\larger
where we have replaced $\cosh AF$ by $$ \cosh AO \cdot \cosh OF-\sinh AO\cdot\sinh OF\cdot \cos \widehat{AOF}.$$
Using the identity,
$$\cosh^2 x-1=\sinh^2 x$$
we simplify the above equality to get
\begin{equation}\label{equ:tan}
\tan \widehat{AFO} = \frac{\sinh AO \cdot\sin\widehat{AOF}}{\cosh AO \cdot \sinh OF-\sinh AO\cdot \cosh OF\cdot \cos \widehat{AOF}}.
\end{equation}

In the same way, we have

\begin{equation}\label{equ:tan2}
\tan \widehat{BFO} = \frac{\sinh BO \cdot\sin\widehat{BOF}}{\cosh BO \cdot \sinh OF-\sinh BO\cdot \cosh OF\cdot \cos \widehat{BOF}}.
\end{equation}
Since $\widehat{AFO} + \widehat{BFO}= \pi$, $\tan \widehat{AFO} + \tan \widehat{BFO} =0$. Then $\eqref{equ:tan}$ and $\eqref{equ:tan2}$ imply:

\begin{eqnarray*}
&& \sinh AO \cdot \cosh BO \cdot \sinh OF \cdot \sin\widehat{AOF} \\
&-& \sinh AO \cdot \sinh BO \cdot \cosh OF\cdot \sin\widehat{AOF} \cdot \cos \widehat{BOF} \\
&+& \sinh BO \cdot \cosh AO \cdot \sinh OF \cdot \sin\widehat{BOF} \\
&-& \sinh BO \cdot \sinh AO \cdot \cosh OF \cdot \sin\widehat{BOF} \cdot \cos \widehat{AOF}=0.
\end{eqnarray*}
Equivalently, we have
\begin{eqnarray*}
&& \sinh AO \cdot \cosh BO \cdot \sinh OF \cdot \sin\widehat{AOF} \\
&+& \sinh BO 
 \cdot  \cosh AO \cdot \sinh OF \cdot \sin\widehat{BOF} \\
&=& \sinh AO \cdot \sinh BO \cdot \cosh OF 
\\
&.& \left( \sin\widehat{AOF} \cdot \cos \widehat{BOF} + \sin\widehat{BOF} \cdot \cos \widehat{AOF} \right) \\
&=& \sinh AO \cdot \sinh BO \cdot \cosh OF \cdot \sin \widehat{BOD},
\end{eqnarray*}
where the last equality follows from $\widehat{AOF} + \widehat{BOF} + \widehat{BOD}= \pi$. This implies
\smaller
\begin{eqnarray*}
\sin \widehat{BOD} 
\end{eqnarray*}
\smaller
\begin{eqnarray*}
= \frac{\sinh AO \cdot \cosh BO \cdot \sinh OF \cdot \sin\widehat{AOF} + \sinh BO \cdot \cosh AO \cdot \sinh OF \cdot \sin\widehat{BOF}}{\sinh AO \cdot \sinh BO \cdot \cosh OF } \\
\end{eqnarray*}
\larger
\begin{eqnarray*}
= \tanh OF \cdot \left( \frac{\sin\widehat{BOF}}{\tanh AO} +  \frac{\sin\widehat{AOF}}{\tanh BO} \right).
\end{eqnarray*}
\larger
\end{proof}

For simplicity, we set $p=\widehat{BOF}, q=\widehat{AOF}$ and $r=\widehat{BOD}$. Then $p+q+r=\pi$.
We write Lemma \ref{lem:sin} as

$$\frac{\sin r}{\tanh OF}=  \frac{\sin p}{\tanh AO} + \frac{\sin q}{\tanh BO}.$$

By repeating the arguments in the proof of Lemma \ref{lem:sin}, we have

$$\frac{\sin p}{\tanh OD}=  \frac{\sin q}{\tanh BO} + \frac{\sin r}{\tanh CO}$$
and

$$\frac{\sin q}{\tanh OE}=  \frac{\sin r}{\tanh CO} + \frac{\sin p}{\tanh AO}.$$

Setting $P= \displaystyle\frac{\sin p}{\tanh AO}, Q=  \frac{\sin q}{\tanh BO}$ and $\displaystyle R= \frac{\sin r}{\tanh CO}$, it follows from the above three equations that

\begin{equation}
\gamma R= P+Q \label{equ:R},
\end{equation}
\begin{equation}\alpha P= Q+R \label{equ:P},
\end{equation}
\begin{equation}
\beta Q= R+P \label{equ:Q}.
\end{equation}

Using \eqref{equ:R} and \eqref{equ:P}, we have

$$R= \frac{P+Q}{\gamma}=\alpha P - Q.$$ As a result,
$$\frac{P}{Q}=\frac{\gamma+1}{\alpha \gamma-1}.$$
On the other hand, it follows from \eqref{equ:P} and \eqref{equ:Q} that

$$\alpha P= Q+R= Q+ \beta Q-P.$$
Then we have

$$\frac{P}{Q}= \frac{\beta+1}{\alpha+1}.$$

Thus,

$$\frac{\gamma+1}{\alpha\gamma-1}=\frac{\beta+1}{\alpha+1},$$
which implies

$$\alpha\beta\gamma=\alpha+\beta+\gamma+2.$$

\begin{remark} In  \cite{Euler-Geometrica-T}, Euler also writes
Equation \eqref{equ:Euler1} as $$\frac{1}{\alpha+1}+ \frac{1}{\beta+1}+ \frac{1}{\gamma+1}=1,$$
and this leads, in the hyperbolic case, to the relation:
\[
\frac{\tanh OD}{\tanh AO+\tanh OD}+ \frac{\tanh OE}{\tanh BO+\tanh OE}+\frac{\tanh OF}{\tanh CO+\tanh OF}=1.
\]

\end{remark}

%

\subsection{The converse}

\begin{construction}From the six given quantities  $\mathbf{A,B,C, a,b,c}$ satisfying the
relation 
\begin{equation}\label{equ:relation}
\alpha\beta\gamma=\alpha+\beta+\gamma+2,
 \end{equation}
 where $$\alpha=\frac{\tanh \mathbf{A}}{\tanh \mathbf{a}}, \beta=\frac{\tanh \mathbf{B}}{\tanh \mathbf{b}}, \gamma=\frac{\tanh \mathbf{C}}{\tanh \mathbf{c}},$$
we can construct a unique triangle $ABC$ in which three line segments
$AD, BE, CF$ are drawn from each vertex to the opposite side, meeting at a point
$O$ and leading to the given arcs:

\begin{eqnarray*}
AO = \mathbf{A}, BO=\mathbf{B}, CO=\mathbf{C}, \\
OD = \mathbf{a}, OE=\mathbf{b}, OF=\mathbf{c}.
\end{eqnarray*}

\end{construction}

The construction is the same as Euler's in the case of a spherical triangle. We
are given three segments $AOD, BOE, COF$ intersecting at a common point $O$, and
we wish to find the angles $ \widehat{AOF},  \widehat{BOF},  \widehat{BOD}$ \ (Figure 2), so that the
three points $A, B, C$ are vertices of a triangle and $D, E, F$ are on
the opposite sides.

Staring from Equation $\eqref{equ:relation}$, we can show (reversing the above reasoning) that the angles $ \widehat{AOF},  \widehat{BOF},  \widehat{BOD}$ should satisfy
\smaller
$$\sin \widehat{BOF}=\tanh \mathbf{A} \cdot \frac{\Delta}{\alpha+1},\  \sin  \widehat{AOF}=\tanh \mathbf{B} \cdot \frac{\Delta}{\beta+1},\  \sin \widehat{BOD}=\tanh \mathbf{C} \cdot \frac{\Delta}{\gamma+1},$$
\larger
where $\Delta>0$ is to be determined.

Setting

$$G=\frac{\tanh \mathbf{A}}{\alpha+1}, H=\frac{\tanh \mathbf{B}}{\beta+1}, I=\frac{\tanh \mathbf{C}}{\gamma+1}$$
and using the fact that the angles satisfy the further equation
$\widehat{AOF} + \widehat{BOF} + \widehat{BOD} = \pi$,
we get (by writing the formula for the sum of two supplementary angles):

$$\Delta=\frac{\sqrt{(G+H+I)(G+H-I)(I+G-H)(H+I-G)}}{2GHI}.$$
Hence, $\Delta$ is uniquely determined. 

Note that  the area $M$ of a Euclidean triangle
with sides $G,H,I$ is given by the following (Heron Formula):
$$M = \Delta \cdot \frac{GHI}{2}.$$

A calculation gives then the following formula for the angles:
$$\sin \widehat{BOF}  = \frac{2M}{HI},\  \sin  \widehat{AOF} = \frac{2M}{IG},\  \sin  \widehat{BOD}  = \frac{2M}{GH}$$
From these angles, we can construct the triangle by drawing the lines $AD, BE$
that intersect at the point $O$ with angle $\widehat{BOD}$. 

\subsection{Another proof of Theorem \ref{thm:hyperbolic}}

We present another proof of Theorem \ref{thm:hyperbolic}, based on the hyperboloid model of the hyperbolic plane. This will make another analogy with the spherical case.

We denote by $\mathbb{R}^{2,1}$ the three-dimensional Minkowski space,
that is, the real vector space of dimension three equipped with the following pseudo-inner product:
$$<\mathbf{x},\mathbf{y}>=-x_0y_0+x_1y_1+x_2y_2.$$

We consider the hypersurface
$$\mathbb{H} := \{\mathbf{x}\in \mathbb{R}^{2+1} \ | \ x_0>0, <\mathbf{x},\mathbf{x}>=-1\}.$$
This is one of the two connected components of the ``unit sphere" in this space, that is, the sphere of radius $\sqrt{-1}$. We shall call this component the \emph{imaginary sphere}.

At every point $\mathbf{x}$ of the imaginary sphere, we equip the tangent space $\mathrm{T}_{\mathbf{x}}\mathbb{H}$ at $\mathbf{x}$ with the pseudo-inner product induced from that on $\mathbb{R}^{2+1}$. It is well known that this induced pseudo-inner product is a scalar product, and the imaginary sphere equipped with the length metric induced from these inner products  on tangent spaces is isometric to the hyperbolic plane. This is a model of the hyperbolic plane, called the Minkowski model. See \cite{Thurston} for some details.

 Let $\mathbf{x},\mathbf{y}$
be two points in $\mathbb{H}$. It is well known and not hard to show that their distance $d(\mathbf{x},\mathbf{y})$ is given by

$$\cosh d(\mathbf{x},\mathbf{y})=-<\mathbf{x},\mathbf{y}>.$$

Up to an isometry of $\mathbb{H}$, we may assume that $\mathbf{x}=(1,0,0)$ and $\mathbf{y}=a\mathbf{x}+b\mathbf{n}$,
where $\mathbf{n}=(0,1,0)$. The equation $<\mathbf{y},\mathbf{y}>=-1$ implies $a^2-b^2=1$. We may also assume that $b\geq 0$.
See Figure \ref{fig:model}.
Then we have
 $$\cosh d(\mathbf{x},\mathbf{y})=-<\mathbf{x},\mathbf{y}>=-a<\mathbf{x},\mathbf{x}>-b<\mathbf{x},\mathbf{n}>=a.$$
 It follows that
 \begin{equation}\label{eq:n}
 \mathbf{y}=\cosh \left(d(\mathbf{x},\mathbf{y})\right)\mathbf{x}+\sinh \left(d(\mathbf{x},\mathbf{y})\right)\mathbf{n}.
 \end{equation}

\begin{figure}[htbp]

\centering

\includegraphics[width=10cm]{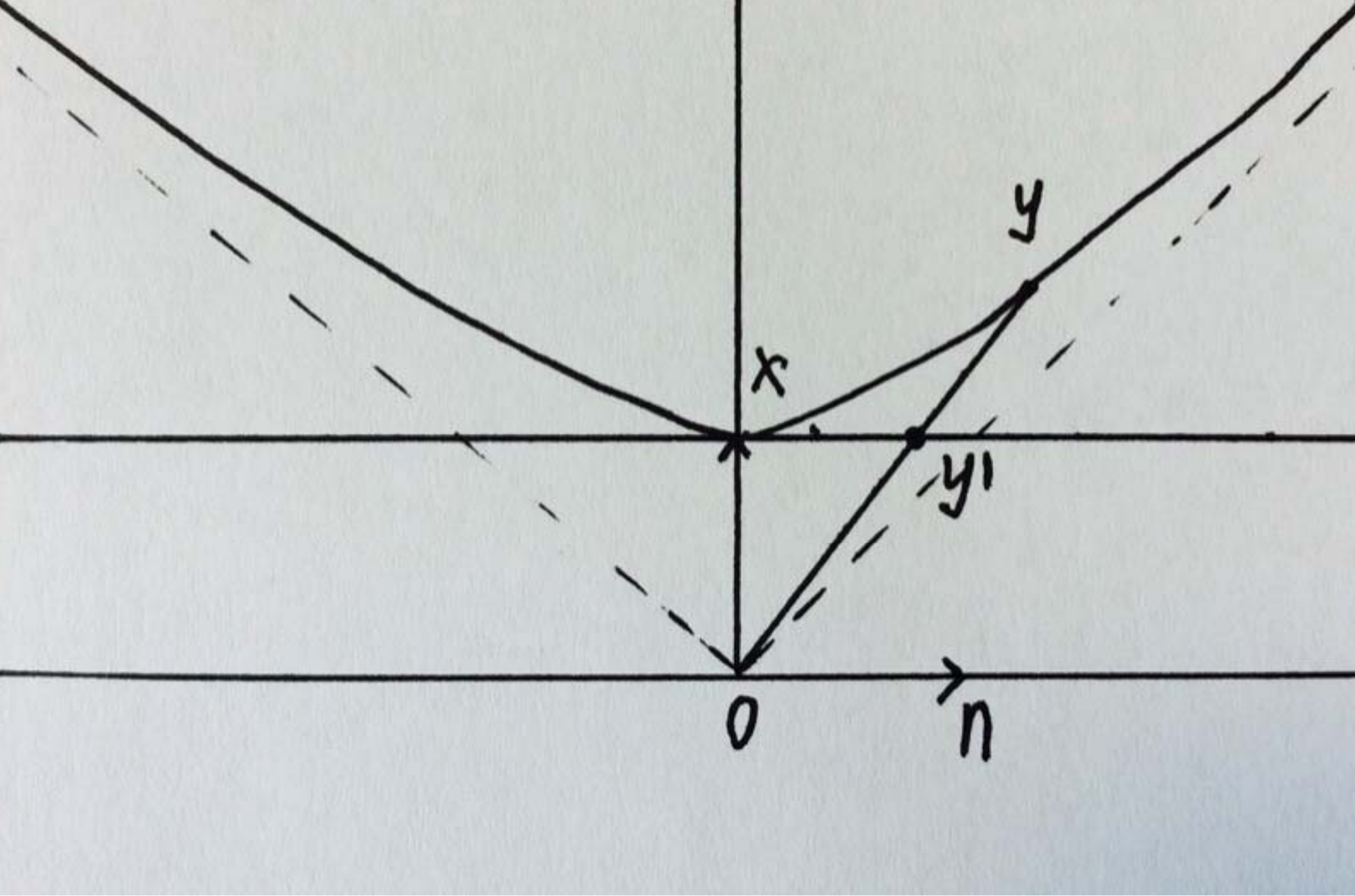}

\caption{The hyperboloid model of the hyperbolic plane.} \label{fig:model}

\end{figure}

\begin{proof}[Another proof of Theorem \ref{thm:hyperbolic}]
Consider a triangle  $ABC$ in $\mathbb{H}$, with $D, E, F$ on the lines $BC, CA, AB$, respectively.
Suppose that the lines $AD, BE$ and $CF$ intersect at $O$. Up to an isometry, we may suppose that the point $O$ is $(1,0,0)$.

Let $\Sigma$ be the plane tangent to $\mathbb{H}$ at $O=(1,0,0)$. We shall use the Euclidean metric on this plane. For any point $\mathbf{y}\in \mathbb{H}$, the line drawn from $\mathbf{0}=(0,0,0)$
through $\mathbf{y}$ intersects $\Sigma$ at a unique point, which we denote by $\mathbf{y}'$. Consider the points $A',B',C',D',E',F'$ obtained from the intersections of the lines $0A,0B,0C,0D,0E,0F$ with $\Sigma$, respectively.
By \eqref{eq:n},
$$\tan\widehat{O\mathbf{0}A}= \frac{OA'}{1}=\frac{\sinh OA}{\cosh OA}.$$
(Here $OA'$ denotes the Euclidean distance between $O$ and $A'$, and $OA$ denotes the hyperbolic distance between $O$ and $A$.)
As a result, $OA'=\tanh OA$. Similarly,  $OB'=\tanh OB,  OC'=\tanh OC$, $OD'=\tanh OD, OE'=\tanh OE$ and $OF'=\tanh OF$.

Since
$$\alpha=\frac{\tanh OA}{\tanh OD}=\frac{OA'}{OD'}, \, \beta=\frac{\tanh OB}{\tanh OE}=\frac{OB'}{OE'},\, \gamma= \frac{\tanh OC}{\tanh OF}=\frac{OC'}{OF'},$$
we have reduced the proof to the case of a Euclidean triangle.
\end{proof}

\begin{remark}
The preceding proof is inspired from an argument that Euler gave in a second proof of his Theorem \ref{thm:Euler2}. Euler's argument uses a radial projection of the sphere onto a Euclidean plane tangent to the sphere, which we have transformed into an argument that uses the radial projection of the imaginary sphere onto a Euclidean plane.
\end{remark}

\subsection{Ceva's theorem} 
 We refer again to Figure \ref{fig:triangle}. 
The classical theorem of Ceva\footnote{Giovanni Ceva (1647-1734) obtained the statement in the Euclidean case, in his \emph{De lineis rectis se invicem secantibus statica constructio}, 1678. According to Hogendijk, Ceva's theorem was already known to the Arabic mathematician Ibn H\=ud, cf. \cite{Hogendijk}.} gives another necessary and sufficient relation for the three lines $AD, BE, CF$ to meet in a point, and it has also Euclidean, spherical and hyperbolic versions. The Ceva identity is different from Euler's. The statement is:

\begin{theorem}
If the three lines $AD, BE, CF$ meet in a common point, then we have \begin{itemize}
\item in Euclidean geometry:
\[\frac{DB}{DC}\cdot\frac{EC}{EA}\cdot\frac{FA}{FB}=1;\] 

\item in spherical geometry:
\[ \frac{\sin DB}{\sin DC}\cdot\frac{\sin EC}{\sin EA}\cdot\frac{\sin FA}{\sin FB}=1.\] 
\item in hyperbolic geometry:
\[ \frac{\sinh DB}{\sinh DC}\cdot\frac{\sinh EC}{\sinh EA}\cdot\frac{\sinh FA}{\sinh FB}=1.\] 
\end{itemize}
\end{theorem}
 
\begin{proof} We give the proof in the case of hyperbolic geometry. The other proofs are similar.
Assume the three lines meet at a point $O$. 
  By the sine law, we have
\[\frac{\sinh DB}{\sin \widehat{DOB}}=\frac{\sinh OB}{\sin \widehat{ODB}}
\]
and 
\[\frac{\sinh DC}{\sin \widehat{DOC}}=\frac{\sinh OC}{\sin \widehat{ODC}}.
\]
Dividing both sides of these two equations, we get:
 
\[\frac{\sinh DB}{\sinh DC}=\frac{\sinh OB}{\sinh OC}\cdot \frac{\sin \widehat{DOB}}{\sin \widehat{DOC}}.
\]
In the same way, we have
\[\frac{\sinh EC}{\sinh EA}=\frac{\sinh OC}{\sinh OA}\cdot \frac{\sin \widehat{EOC}}{\sin \widehat{EOA}}
\]
and
\[\frac{\sinh FA}{\sinh FB}=\frac{\sinh OA}{\sinh OB}\cdot \frac{\sin \widehat{FOA}}{\sin \widehat{FOB}}.
\]
Multiplying both sides of the last three equations and using the relations
\[\sin \widehat{DOB}= \sin \widehat{EOA}, \ 
 \sin \widehat{DOC}= \sin \widehat{FOA}, \ 
 \sin \widehat{EOC}= \sin \widehat{BOF},  
\]
we get the desired result.
 \end{proof}

  \begin{remark} The classical theorem of Ceva is usually stated with a minus sign at the right hand side (that is, the result is $-1$ instead of $1)$, and the length are counted algebraically. In this form, the converse of the theorem holds. The proof is also easy.
  \end{remark}

\subsection{A theorem of Lambert}

We present now a result of Lambert,\footnote{Johann Heinrich Lambert  (1728-1777) was an Alsatian mathematician (born in Mulhouse). He is sometimes considered as the founder of modern cartography, a field which was closely related to spherical geometry. His \emph{Anmerkungen und Zus\"atze zur Entwerfung der Land- und Himmelscharten} (Remarks and complements for the design of terrestrial and celestial maps, 1772) \cite{Lamb-Anmer} contains seven new projections of the sphere, some of which are still in use today, for various purposes. Lambert is an important precursor of hyperbolic geometry; he was probably the mathematician who came closest to that geometry, before this geometry was born in the works of Lobachevsky, Bolyai and Gauss. In his {\it Theorie der Parallellinien}, written in 1766, he developed the bases of a geometry in which all the Euclidean postulates hold except the parallel postulate which is replaced by its negation. His hope was to arrive to a contradiction, which would show that Euclid's parallel postulate is a consequence of the other Euclidean postulates. Instead of leading to a contradiction, Lambert's work turned out to be a collection of results in hyperbolic geometry, to which belongs the result that we present here. We refer the reader to \cite{Lambert-Blanchard} for the first translation of this work originally written in old German, together with a mathematical commentary. Lambert was self-taught (he left school at the age of eleven), and he eventually became one of the greatest and most universal minds of the eighteenth century. Euler had a great respect for him, and he helped him joining the Academy of Sciences of Berlin, where Lambert worked during the last ten years of his life. One of Lambert's achievements is that $\pi$ is irrational. He also conjectured that $\pi$ is transcendental (a result which was obtained a hundred years later).} contained in his \emph{Theory of parallel lines}, cf. \cite{Lambert-Blanchard}, \S 77. This result says that in an equilateral triangle $ABC$, if $D$ is the midpoint of $BC$ and $O$ the intersection point of the medians, we have $OD=\frac{1}{3} AD$,  $OD>\frac{1}{3} AD$, $OD<\frac{1}{3} AD$   in Euclidean, spherical, hyperbolic geometry  respectively.
In fact, we shall obtain a more precise relation between the lengths involved. The hyperbolic case will follow from the following proposition:
 
\begin{proposition}
Let  $ABC$ be an equilateral triangle in the hyperbolic plane and let $D,E,F$ be the midpoints of $BC, AC, AB$, respectively (Figure \ref{fig:triangle}). Then the
lines $AD, BE, CF$ intersect at a common point $O$ satisfying
$$\frac{\tanh AO}{\tanh OD}=\frac{\tanh BO}{\tanh OE}=\frac{\tanh CO}{\tanh OF}=2.$$
In particular, 
$$\frac{AD}{OD}=\frac{BE}{OE}=\frac{CF}{OF}>3.$$
\end{proposition}

\begin{proof}
The fact that the
lines $AD, BE, CF$ intersect at a common point $O$ follows from the symmetry of the equilateral triangle. 

Let us set
$$\alpha=\frac{\tanh AO}{\tanh OD},\beta=\frac{\tanh BO}{\tanh OE},\gamma=\frac{\tanh CO}{\tanh OF}.$$
Then, again by symmetry, $\alpha=\beta=\gamma$. By the hyperbolic version of Euler's Theorem (Theorem \ref{thm:hyperbolic}), we have
$$\alpha^3=3\alpha+2.$$
This implies that $\alpha=\beta=\gamma=2$.

To see that $\frac{AD}{OD}>3$ (or, equivalently, $\frac{AO}{OD}>2$), it suffices to
check that $\tanh AD=2\tanh OD> \tanh (2 OD)$. This follows from the inequality
$$2\tanh x > \tanh (2x), \forall \ x>0.$$
\end{proof}

An analogous proof shows that in the Euclidean case, and with the same notation, we have 
$\frac{AD}{OD}=3$ and in the spherical case, we have $\frac{AD}{OD}<3$. 

\bigskip

\section{Hyperbolic triangles with the same area}

In this section, we will study the following question:

\begin{quote}\emph{Given two distinct points $A,B\in \mathbb{H}^2$, determine the set of points $P\in \mathbb{H}^2$ such that the area of the triangle with vertices $P,A, B$ is equal to some given constant.}
\end{quote}

The question in the case of a spherical triangle was solved by Lexell \cite{Lexell-Solutio} and Euler \cite{Euler-Variae-T}.\footnote{Despite the difference in the dates of publication, the papers of Euler and Lexell were written the same year.}
We provide a proof for the case of a hyperbolic triangle.

Let us note that the analogous locus in the Euclidean case is in Euclid's \emph{Elements} (Propositions 37 and its converse, Proposition 39, of Book I). In this case, the locus consists of a pair of lines parallel to the basis. In spherical and hyperbolic geometries, the locus does not consist of lines (that is, geodesics) but of hypercycles (equidistant loci to lines) that pass by the points antipodal to the basis of the triangle. Also note that these hypercycles are not equidistant to the line containing the base of the triangle. The two hypercycles are equidistant to two distinct lines. 

This theorem in the spherical case, has an interesting history. Both Euler and his student Lexell gave a proof in \cite{Euler-Variae-T} (published in 1797) and \cite{Lexell-Solutio} (published in 1784).\footnote{We already noted that the two memoirs were written in the same year. Euler says that the idea of the result was given to him by Lexell.} Jakob Steiner published a proof of the same theorem in 1827 \cite{Steiner}, that is, several decades after Euler and Lexell. In 1841, Steiner published a new proof \cite{Steiner3}. In the same paper, he says that Liouville, the editor of the journal in which the paper appeared, and before he presented the result at the  Academy of Sciences of Paris, looked into the literature and found that Lexell already knew the theorem.  Steiner mentions that the theorem was known, ``at least in part", to Lexell, and then to Legendre. He does not mention Euler. Steiner adds:  ``The application of the theorem became easy only after the following complement: \emph{the circle which contains the   triangles with the same area passes through the points antipodal to the extremities of the bases}." In fact, this ``complement" is contained in Lexell's proof.  Legendre gives a proof of the same theorem in his \emph{\'El\'ements de g\'eom\'etrie}, \cite{Legendre} Note X, Problem III. His solution is based on spherical trigonometry, like one of Euler's. In 1855, Lebesgue gave a proof of this theorem \cite{Lebesgue1855}, which in fact is Euler's proof. At the end of Lebesgue's paper, the editor of the journal adds a comment, saying that one can find a proof of this theorem in the \emph{\'El\'ements de G\'eom\'etrie} of Catalan (Book VII, Problem VII), but no reference is given to Euler.

In the rest of this section, we prove the hyperbolic analogue of this theorem. We shall use the unit disc model of the hyperbolic plane. Up to an isometry, we may assume that the two vertices
$A$ and $B$ lie on the real line, with $0<A=-B<1$ (that is, $A$ and $B$ are symmetric with respect to the origin). This will simplify the notation and will make our discussion clearer. We assume that the hyperbolic distance between $A$ and $B$ is of the form $2x$.

\subsection{Example: A family of triangles with increasing areas}\label{sec:example}

Let us denote the center of the unit disc by $O$.
We first consider the case where the vertex $P$ lies on the geodesic that goes through $O$ perpendicularly to the real line. We denote the
hyperbolic distance between $O$ and $P$ by $y$, and the hyperbolic distance between $A$ and $P$ by $c$.  See Figure \ref{fig:apb}.

\begin{figure}[htbp]

\centering

\includegraphics[width=10cm]{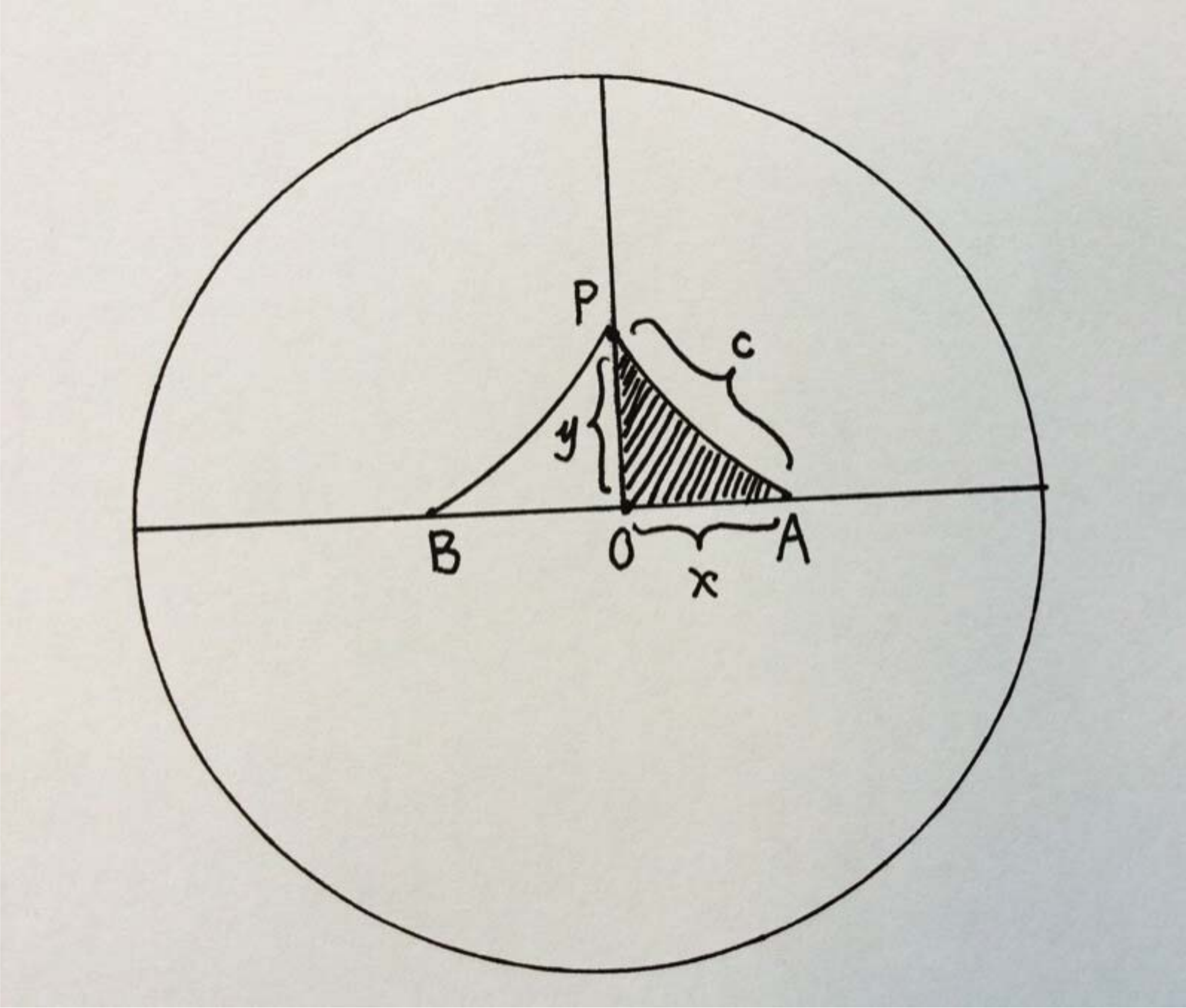}

\caption{The example $APB$} \label{fig:apb}

\end{figure}

Since the distances $x=d(O,A)=d(O,B)$ are fixed, we may
consider the area of the triangle $APB$ as a function of $y$. This area is the double of the area of $APO$.  We start with the following:

\begin{proposition}\label{Pro:area}
The area of $APO$
is an increasing function of $y$.
\end{proposition}
\begin{proof} We use hyperbolic trigonometry. 
By the cosine law for a hyperbolic triangle, we have
$$\cos \widehat{AOP} = \frac{\cosh x \cdot \cosh y-\cosh c}{\sinh x \cdot \sinh y}.$$
Since $\widehat{AOP}=\frac{\pi}{2}$, $\cos\widehat{AOP}=0$. We have

$$\cosh c=\cosh x \cdot \cosh y.$$

Denote $\widehat{APO}=\alpha, \widehat{PAO}=\beta$. Using  again the cosine law for hyperbolic triangles, we obtain

$$\cos \alpha = \frac{\cosh y \cdot \cosh c - \cosh x}{\sinh y \cdot \sinh c}, \
\cos \beta = \frac{\cosh x \cdot \cosh c - \cosh y}{\sinh x \cdot \sinh c}.$$

Using the sine law for hyperbolic triangles, we have (note that $\sin \widehat{AOP}=1$)

$$\sin \alpha= \frac{\sinh x}{\sinh c}, \  \sin \beta = \frac{\sinh y}{\sinh c}.$$

Applying the above equations, we have

\begin{eqnarray*}
\sin(\alpha+\beta)&=& \sin \alpha \cdot \cos\beta + \sin\beta \cdot \cos\alpha \\
&=& \frac{\sinh x}{\sinh c} \cdot \frac{\cosh x \cdot \cosh c - \cosh y}{\sinh x \cdot \sinh c} +
\frac{\sinh y}{\sinh c} \cdot \frac{\cosh y \cdot \cosh c - \cosh x}{\sinh y \cdot \sinh c} \\
&=& \frac{(\cosh c-1)(\cosh x + \cosh y)}{(\sinh c)^2} \\
&=&  \frac{(\cosh x \cdot \cosh y-1)(\cosh x + \cosh y)}{(\cosh x \cdot \cosh y)^2-1}. \\
\end{eqnarray*}

We set $u=\cosh y>1$ and write the right-hand side of the above equation as

\begin{eqnarray*}
f(u) &=&  \frac{(\cosh x \cdot u-1)(\cosh x + u)}{(\cosh x \cdot u)^2-1} \\
&=& \frac{\cosh x \cdot u^2 + (\sinh x)^2 \cdot u -\cosh x}{(\cosh x)^2 \cdot  u^2 -1}.
\end{eqnarray*}

By a calculation, we have

$$f'(u)=-(\sinh x)^2  \cdot \frac{1}{ (\cosh x \cdot  u+1)^2}< 0.$$

This shows that $\sin(\alpha+\beta)$ is a decreasing function of $y$. Since the area of $APO$ is
given by $\frac{\pi}{2}-(\alpha+\beta)$, it is an increasing function of $y$.
\end{proof}

More generally, consider, instead of the positive $y$-axes,  an  arbitrary geodesic ray $\Gamma(t),t\in [0,\infty)$ initiating from the real line perpendicularly.
Geodesics are parameterized by arc-length
We denote by $F\in (-1,1)$ the initial point of $\Gamma(t)$, i.e. $F=\Gamma(0)$. This point is the foot of $\Gamma(t)$ on the real line, for any $t\in [0,\infty)$. We denote by $a$ the (hyperbolic) distance between $O$ and $F$.
Let us assume for the proof that $F\in [B,A]$. The area of the triangle is the sum of the areas of
$AF\Gamma(t)$ and $BF\Gamma(t)$. The proof of  Proposition  \ref{Pro:area} applied to each of these triangles   shows that the area of the hyperbolic triangle with vertices $A, B, \Gamma(t)$
is an increasing function of $t$.

The triangle $AB\Gamma(t)$ is naturally separated  by $F\Gamma(t)$ into two right triangles. (When $F$ coincides with $A$ or $B$,
we consider that one of the triangles is of area $0$.) Denoting by $\Delta_1(t)$ and $\Delta_2(t)$ the areas of these two right
triangles, we have

\begin{equation}\label{equ:area}
\Delta_1(t)=\arccos \left(  \frac{\left(\cosh (x-a) \cdot \cosh t-1)(\cosh (x-a) + \cosh t\right)}{\left(\cosh (x-a) \cdot \cosh t\right)^2-1} \right)
\end{equation}

\begin{equation}\label{equ:area}
\Delta_2(t)=\arccos \left(  \frac{\left(\cosh (x+a) \cdot \cosh t-1)(\cosh (x+a) + \cosh t\right)}{\left(\cosh (x+a) \cdot \cosh t\right)^2-1}\right).
\end{equation}

We have shown that both $\Delta_1(t)$ and $\Delta_2(t)$ are strictly increasing functions of $t$.
By making $t\to\infty$, we have
$$\lim_{t\to\infty}\left( \Delta_1(t)+\Delta_2(t)\right)= \arccos \left( \frac{1}{\cosh (x-a)}\right) + \arccos \left( \frac{1}{\cosh (x+a)}\right).$$
The limit is the area of the ideal triangle with vertices $A,B,\Gamma(\infty)$.

\subsection{The locus of vertices of triangles with given base and area}

As in \S \ref{sec:example}, we assume that $A,B\in \mathbb{H}$ are on the real line and are symmetric with respect to
the imaginary axis. For any $P\in \mathbb{H}$, we denote by $\Delta(P)$ the area of the hyperbolic triangle  $APB$.
Let $\mathcal{L}(P)$  be the set of points $Z$ in $\mathbb{H}$ such that the area of the triangle $AZB$ is equal to $\Delta(P)$.

We recall that a \emph{hypercycle} $\mathcal{C}$ in $\mathbb{H}$  is a bi-infinite curve in $\mathbb{H}$
 whose points are equidistant from a given geodesic.
In the unit disc model of the hyperbolic plane, $\mathcal{C}$ is represented by an arc of circle that intersects the boundary circle at non-right angles. The horocycle $\mathcal{C}$ and its associated geodesic intersect the boundary circle in the same points, and the geodesic makes right angles with the unit circle at these points.  The angle that $\mathcal{C}$ makes with the unit circle is right if and only if this hypercycle coincides with the associated geodesic.  We shall denote the geodesic associated to $\mathcal{C}$ by $\mathcal{G}$. There is another hypercycle, on the other side of $\mathcal{G}$, with the same distance, which will be denoted by $\mathcal{C}'$. We need only consider hypercycles that
are symmetric with respect to the imaginary axes.

With the above notation, we can state our main result.

\begin{theorem}\label{thm:area}
For any $P\in \mathbb{H}$, there is a unique hypercycle $\mathcal{C}$ that passes through $A,B$ such that
$\mathcal{C}'$ is one of the two connected components of the locus $\mathcal{L}(P)$ of vertices $Z$ of triangles $ABZ$ having the same area as $ABP$. 
\end{theorem}

Theorem \ref{thm:area} is illustrated in Figure \ref{fig:cycle}.

\begin{figure}[htbp]

\centering

\includegraphics[width=10cm]{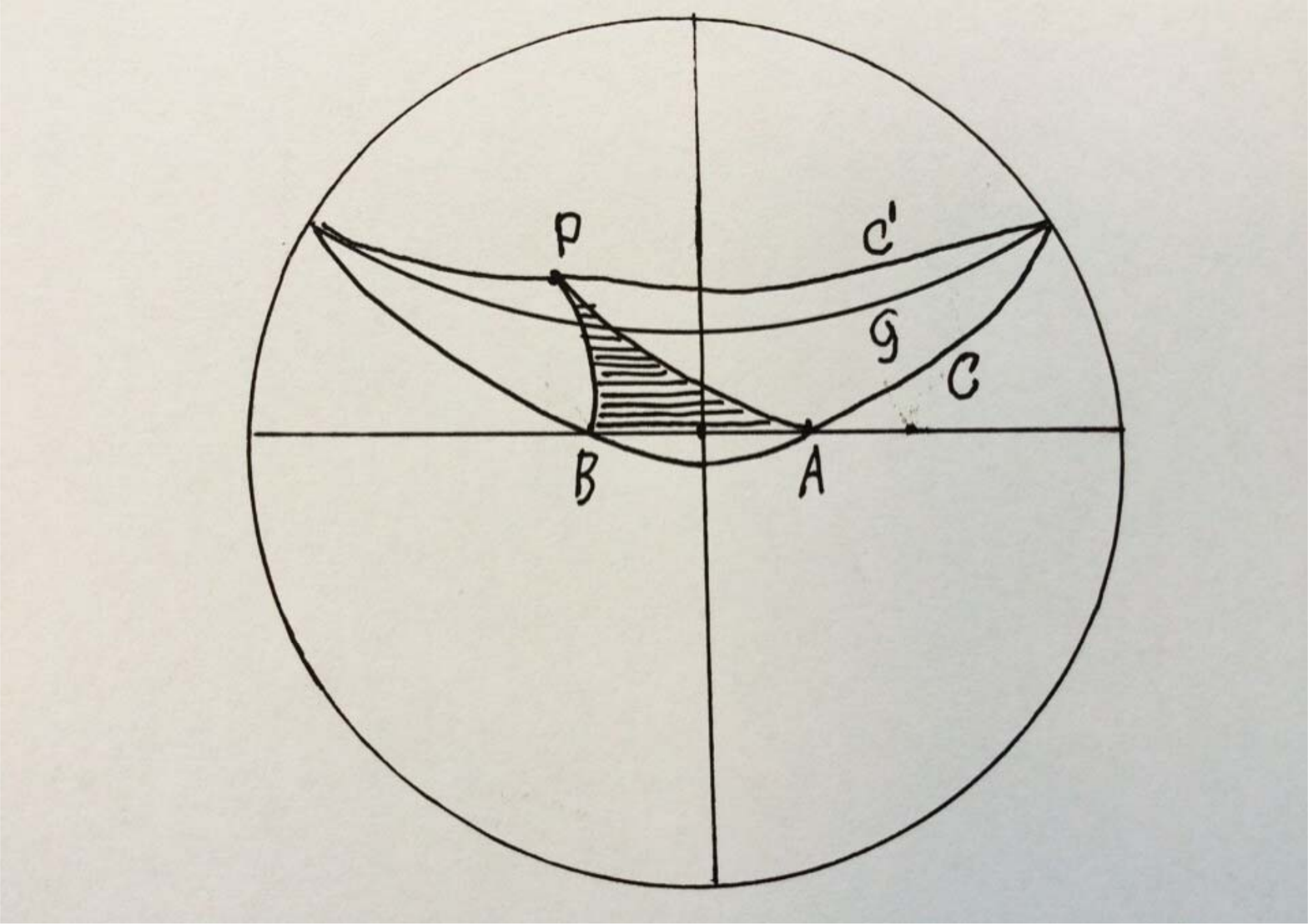}

\caption{The hypercycle $\mathcal{C}$ and $\mathcal{C}'$. When the point $P$ describes the hypercycle $\mathcal{C}'$, the area $\Delta(P)$ is a constant.} \label{fig:cycle}

\end{figure}

\begin{figure}[htbp]

\centering

\includegraphics[width=8cm]{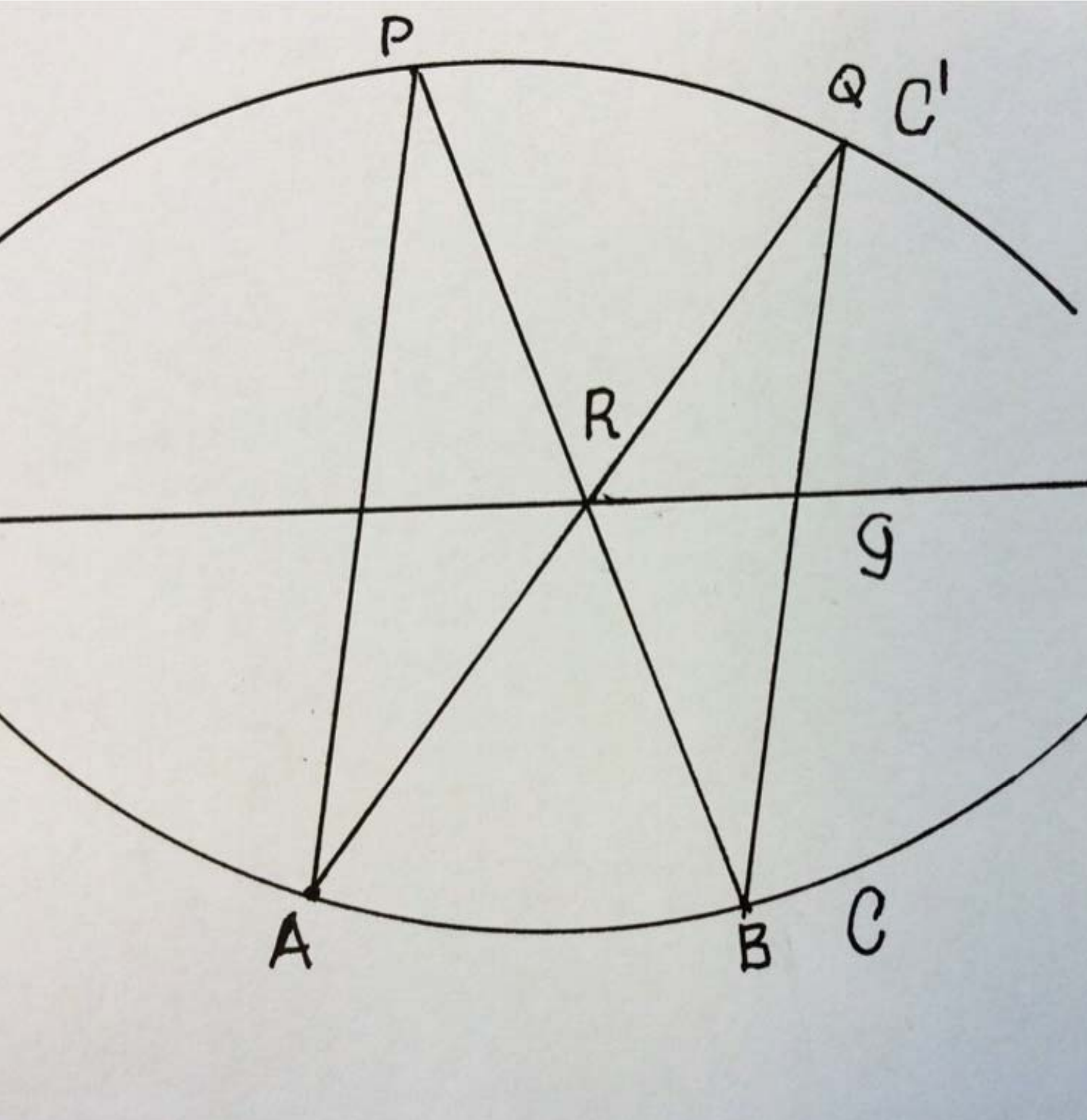}

\caption{The triangle $PRA$ and $QRB$ have the same area.} \label{fig:area}

\end{figure}

\begin{figure}[htbp]

\centering

\includegraphics[width=10cm]{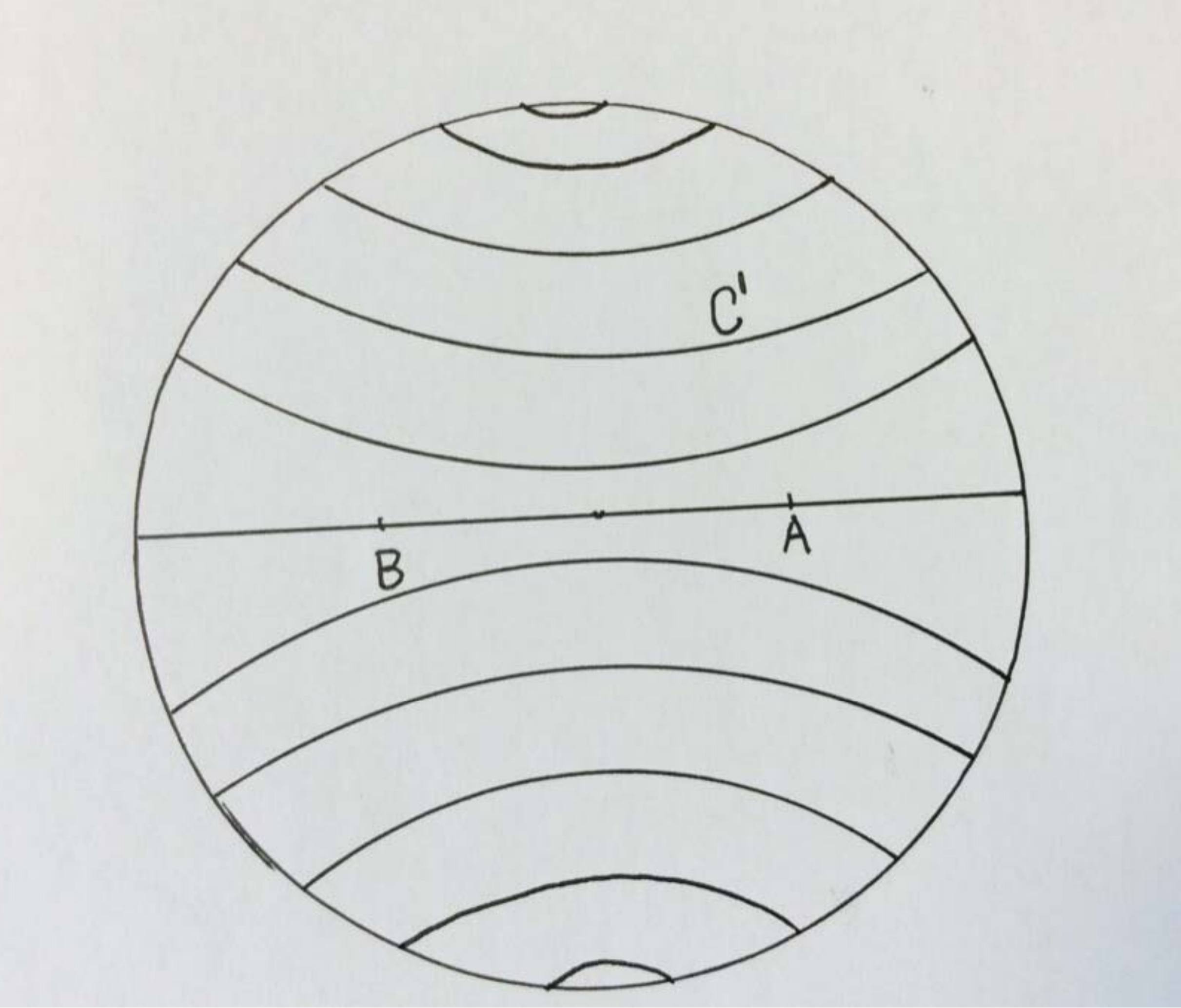}

\caption{The foliation consists of leaves as locus of vertices with the same triangle area.} \label{fig:foliation}

\end{figure}

\begin{proof}[Proof of Theorem \ref{thm:area}]
Consider any hypercycle $\mathcal{C}$ passing through $A,B$ and intersecting the imaginary axis perpendicularly.
As noticed before, there is a unique geodesic $\mathcal{G}$ equidistant to $\mathcal{C}$. There is another hypercycle $\mathcal{C}'$ that is symmetric to $\mathcal{C}$ with respect to $\mathcal{G}$.

It is a simple fact that any geodesic arc connecting a point in $\mathcal{C}$ and a point in $\mathcal{C}'$ is cut by $\mathcal{G}$ into two sub-arcs of the same hyperbolic length.  In particular, as shown in Figure \ref{fig:area},
for any two points $P, Q$ on $\mathcal{C}'$, the geodesic arcs $PA$ (and also $QB, QA, PB$) are separated by $\mathcal{G}$ into equal segments.
It follows that the area of $PRA$ is equal to the area of $QRB$. Here $R$ denotes the intersection of $PA$ and $QB$,
which necessarily lies on $\mathcal{G}$.
It is also not hard to see that the triangles $PAB$ and $QAB$ have the same area, that is, $\Delta(P)=\Delta(Q)$. The reader may refer to Theorem 5.9 of \cite{ACP}. (Notice that some of the sides of the triangles in Figure \ref{fig:area} on which we are reasoning are pieces of hypercycles instead of being geodesic segments, but the equality between areas remains true when such a piece of horocycle is replaced by the geodesic segment that joins its endpoints.)
Since the points $P,Q$ are arbitrarily chosen on $\mathcal{C}'$, we conclude that the area $\Delta(P)$ of the triangle $PAB$, with vertex  $P$ varying on $\mathcal{C}'$, is constant.

If the hypercycle $\mathcal{C}'$ intersects the imaginary axis at a point which is at distance $y$ from the center of the unit disc, then we showed in \S \ref{sec:example} that $\Delta(P)$, for any $P\in \mathcal{C}'$, is equal to

$$ 2\arccos \left( \frac{(\cosh x \cdot \cosh y-1)(\cosh x + \cosh y)}{(\cosh x \cdot \cosh y)^2-1}\right),$$
where $x$ is as in \ref{sec:example} the distance from $A$ to $O$.
Moreover, we showed  that the area is a increasing function of $y$.

With the above description, we have a fairly clear picture of the locus of vertices with the same triangle area.
For if we move the hypercycle $\mathcal{C}$ continuously in that disc, we get a family of hypercycles $\mathcal{C}'$. Such a family forms a foliation filling the unit disc. On each leaf, the area $\Delta(\cdot)$ is constant. On any two distinct leaves
which are not symmetric with respect to the real axes, the areas are different.
The theorem follows since any point $P\in \mathbb{H}$ lies on a unique leaf of the foliation, and the locus $\mathcal{L}(P)$ consists of
two components, which are symmetric with respect to the real axes.
\end{proof}

\begin{remark}
A limiting case is when the points $A$ and $B$ are both on the ideal boundary of the unit disc. In this case, it is easy to see that on any
hypercycle $\mathcal{C}$ asymptotic to the geodesic $AB$, the area $\Delta(\cdot)$ is constant.

In fact, given such a hypercycle and an arbitrary point $P$ on it, the perpendicular distance between $P$ and $AB$ is a constant $c$ and the geodesic arc realizing the distance between $P$ and $AB$ divides the triangle $PAB$ into two isometric right triangles,
each of which has area

\[\frac{\pi}{2}-\arctan\left( \frac{1}{\sinh c} \right).
\]
This formula follows from the following formula

It suffices to check the following formula (written in Figure \ref{fig:limit}):

\begin{equation}\label{eq:cos}
\sinh c=\frac{\cos\alpha+\cos\beta}{\sin\alpha \cdot \sin\beta}
\end{equation}

To see this, we recall the cosine law for a hyperbolic triangle  with angles $\alpha, \beta, \gamma=0$:

$$\cosh c= \frac{1+\cos\alpha \cdot \cos\beta}{\sin\alpha \cdot \sin\beta}$$
This implies that
\begin{eqnarray*}
\cosh^2c-1 &=& \frac{(1+\cos\alpha \cdot \cos\beta)^2}{\sin^2\alpha \cdot \sin^2\beta}-1 \\
&=& \frac{1+ \cos^2\alpha \cdot \cos^2\beta-\sin^2\alpha \cdot \sin^2\beta+ 2 \cos\alpha \cdot \cos\beta}{\sin^2\alpha \cdot \sin^2\beta}\\
&=&  \frac{1+ \cos^2\alpha \cdot \cos^2\beta-(1-\cos^2\alpha) \cdot (1-\cos^2\beta)+ 2 \cos\alpha \cdot \cos\beta}{\sin^2\alpha \cdot \sin^2\beta} \\
&=&  \frac{ (\cos\alpha + \cos\beta)^2}{\sin^2\alpha \cdot \sin^2\beta}
\end{eqnarray*}
As a result, we have Equation \ref{eq:cos}.

In the limiting case considered, the foliation whose leaves are loci of vertices with the same triangle area consists of hypercycles asymptotic to the geodesic $AB$,
see Figure \ref{fig:limit}. This foliation can be seen as a limit of the foliations constructed in the proof of Theorem \ref{thm:area}.
\end{remark}

\begin{figure}[htbp]

\centering

\includegraphics[width=10cm]{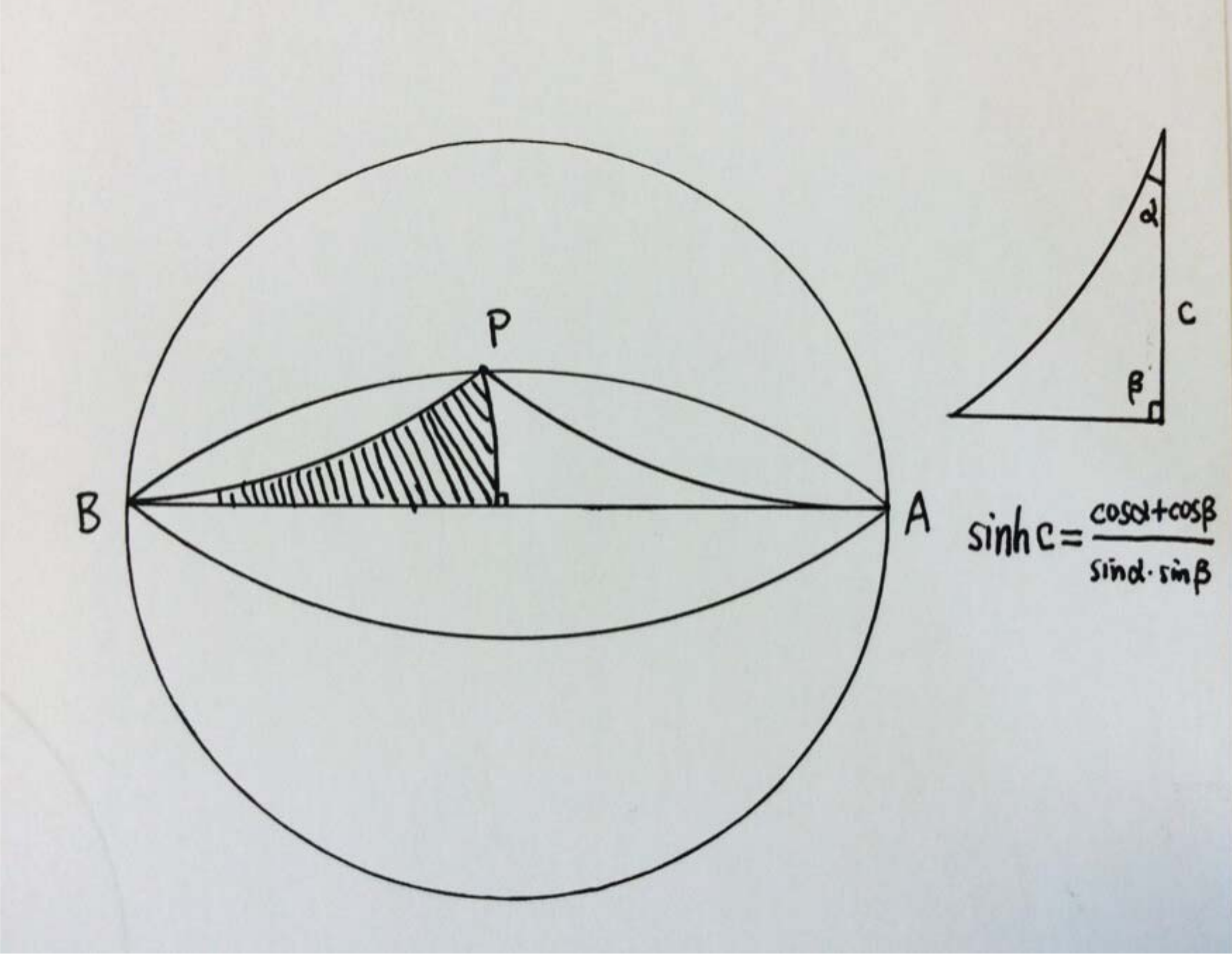}

\caption{The limiting case.} \label{fig:limit}

\end{figure}

\subsection{How to determine the hypercycle}
We conclude by setting up a construction of the hypercycle passing through $P$.

As illustrated in Figure \ref{fig:find}, when $P$ lies on the imaginary axis, we draw the geodesics from $P$ to $A$ and $B$.
To determine the hypercycle through $P$, we only need to determine the midpoint of the geodesics $PA$ and $PB$.
There is a unique geodesic passing through these two midpoints, and it has two endpoints on the ideal boundary. The hypercycle
with the same endpoints is the one we want.

When $P$ does not lie on the imaginary axis, we take the point $P'$ which is symmetric to $P$ with respect to the imaginary axes.
We draw the geodesics connecting $P$ to $A$ and  $P'$ to $B$.
To determine the hypercycle through $P$, we only need to determine the midpoints of the geodesics $PA$ and $P'B$.
There is a unique geodesic passing through the two midpoints, with two endpoints on the ideal boundary. The hypercycle
with the same endpoints is the one we want.

\begin{figure}[htbp]

\centering

\includegraphics[width=10cm]{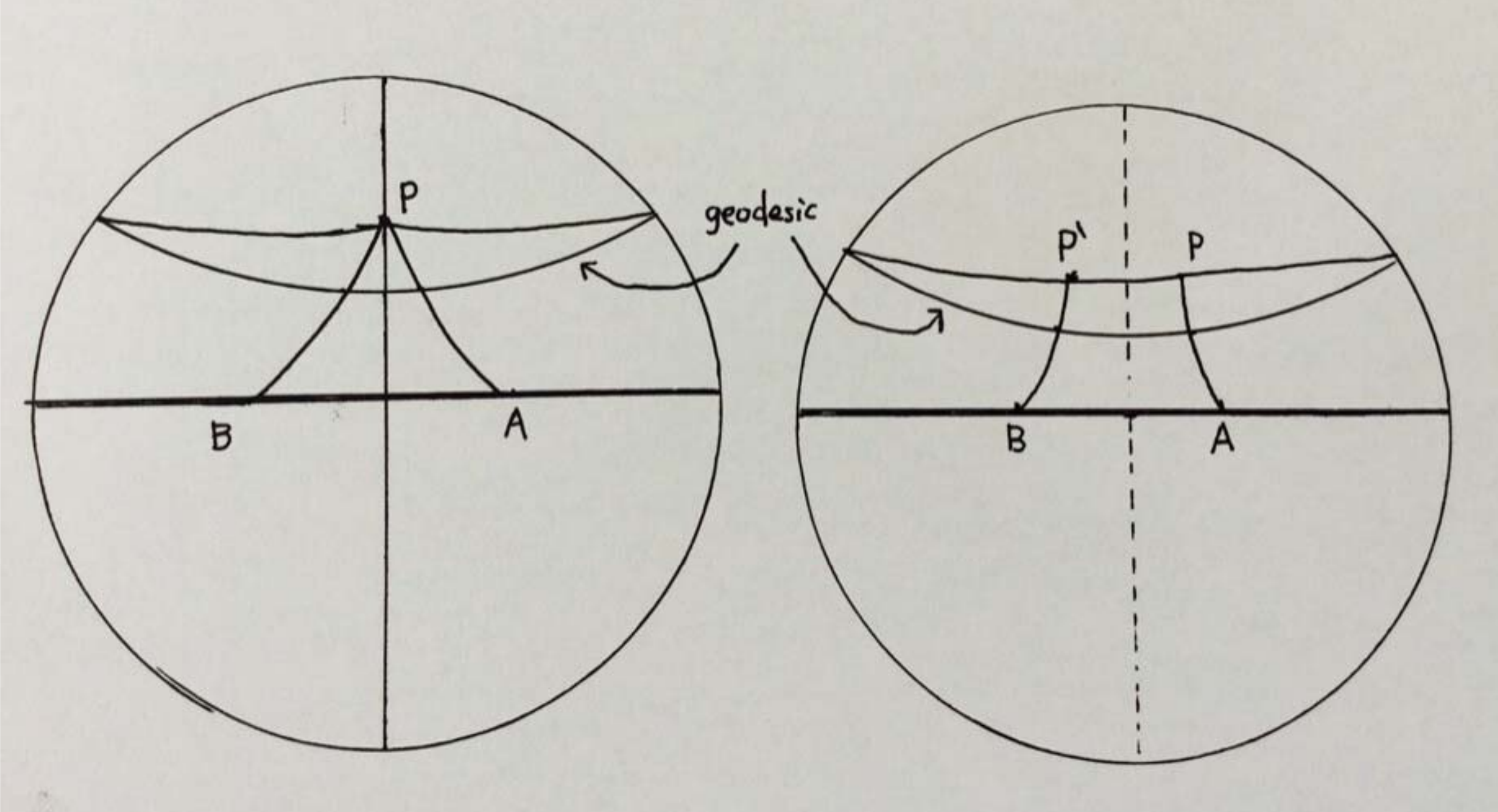}

\caption{To find the hypercycle.} \label{fig:find}

\end{figure}

We then have the following complement to Theorem \ref{thm:area}:

\begin{proposition}
The midpoints of the variable triangles that are on a given side of the line joining $A,B$ are all on a common line, and the locus of the vertices that we are seeking for is a hypercycle with basis that line.
\end{proposition}

In the spherical case the solution is different. The locus of the vertices of the triangles having a given basis and a given area  is a horocycle passing through the points antipodal to the extremities of the bases; this fact has no immediate analogue in the hyperbolic case.

Finally, we mention the following problem, which was first suggested to us by Norbert A'Campo:

\begin{problem}
Work out the three-dimensional analogue of Lexell's theorem, in the spherical and the hyperbolic  cases.
\end{problem}

The theorems we present in this paper are elementary in the sense that their proofs use basic geometry and no advanced theories. Concerning elementary geometrical questions, Klein writes in his \emph{Lectures on mathematics} \cite{Klein-Lectures} (p. 36): ``[...] This is really a question of elementary geometry; and it is interesting to notice how often in recent times higher research has led back to elementary problems not previously settled."  As a matter of fact, the question to which he refers concerns spherical geometry.

 Independently of the proper interest of the theorems presented, we hope that this paper can motivate the reader to read the original sources.

\end{document}